\documentclass[a4paper,reqno]{amsart}
\usepackage{amssymb, amsmath, amscd}
\newtheorem{theorem}{Theorem}[section]
\newtheorem{definition}[theorem]{Definition}
\newtheorem{lemma}[theorem]{Lemma}
\newtheorem{remark}[theorem]{Remark}

\def\Spec{\operatorname{Spec}}\def\Id{\operatorname{Id}}
\def\Span{\operatorname{Span}}\def\Tr{\operatorname{Tr}}
\def\Id{\operatorname{Id}}
\makeatletter
 \@addtoreset{equation}{section}
\begin{document}
\title{Affine projective Osserman structures}
\author{P.  Gilkey  \text{and} S. Nik\v cevi\'c}
\address{PG: Mathematics Department, \; University of Oregon, \;\;
  Eugene \; OR 97403 \; USA}
\email{gilkey@uoregon.edu}
\address{SN: Mathematical Institute, Sanu, Knez Mihailova 36, p.p. 367,
11001 Belgrade, Serbia}
\email{stanan@mi.sanu.ac.rs}
\subjclass[2000]{53C50, 53C44}
\keywords{affine Osserman, affine projective Osserman, spacelike projective Osserman, timelike projective Osserman}
\begin{abstract}
By considering the projectivized spectrum of the Jacobi operator, we introduce the concept of projective 
Osserman manifold in both the affine and in the pseudo-Riemannian settings.
If M is an affine projective Osserman manifold, then the modified Riemannian extension metric on the cotangent 
bundle is both spacelike and timelike projective Osserman. Since any rank 1 symmetric space 
is affine projective Osserman, this
provides additional information concerning the cotangent bundle of a rank 1 Riemannian symmetric space 
with the modified Riemannian extension metric. We construct other examples of affine projective 
Osserman manifolds where the
Ricci tensor is not symmetric and thus the connection in question is not the Levi-Civita connection of any metric.
If the dimension is odd, we use methods of
algebraic topology to show the Jacobi operator of an affine projective Osserman manifold
has only one non-zero eigenvalue and that eigenvalue
is real.
\end{abstract}
\maketitle
\section{Introduction}
\subsection{Osserman geometry in the Riemannian setting}
Let $\mathcal{R}$ be the curvature operator of a Riemannian manifold $\mathcal{M}:=(M,g)$ of
dimension $m$. The Jacobi operator $\mathcal{J}(x):y\rightarrow\mathcal{R}(y,x)x$
is a self-adjoint endomorphism of the tangent bundle. Following the seminal work of 
Osserman \cite{O90}, one says that $\mathcal{M}$
is {\it Osserman} if the eigenvalues of $\mathcal{J}$ are constant on the unit sphere bundle 
$$S(M,g):=\{\xi\in TM:g(\xi,\xi)=1\}\,.$$
Work of Chi \cite{Chi}, of Gilkey, Swan, and Vanhecke \cite{GSV95},
and of Nikolayevsky \cite{Ni1,Ni2}  shows  that any complete and simply connected
Osserman manifold of dimension $m\ne16$
is a rank $1$-symmetric space; the $16$ dimensional setting is exceptional and the situation is
still not clear in that setting although there are some partial results due, again, to Nikolayevsky \cite{Ni06}.

There has been much activity recently in Osserman Geometry. Brozos-V\'{a}zquez and E. Merino \cite{BM12}
showed that in dimension 4, the Osserman condition and the Raki\'c duality principle are equivalent.
Nikolayevsky \cite{Ni12} showed that a conformally Osserman manifold (here one uses
 the Weyl conformal tensor to define the Jacobi operator) is locally isometric to a rank-one symmetric space
in dimension 16 modulo a certain assumption on algebraic curvature tensors in dimension 16. 
Brozos-V\'{a}zquez et. al \cite{BVV09} have examined conformally Osserman manifolds using warped product structures.

\subsection{Osserman geometry in the pseudo-Riemannian geometry}
Suppose that $\mathcal{M}=(M,g)$ is a pseudo-Riemannian manifold
of signature $(p,q)$ for $p>0$ and $q>0$.  The pseudo-sphere bundles are defined by setting: 
$$
S^\pm(M,g)=\{\xi\in TM:g(\xi,\xi)=\pm1\}\,.
$$
One says that $(M,g)$ is spacelike (resp. timelike) Osserman
if the eigenvalues of $\mathcal{J}$ are constant on $S^\pm(M,g)$.
The situation is rather different here as the Jacobi operator
is no longer diagonalizable and can have nontrivial Jordan normal form 
as shown by Garc\'{\i}a-R\'{\i}o et al \cite{GGNV13};
in the algebraic context, the Jordan normal form can be arbitrarily complicated \cite{G-I-2}.
One says $(M,g)$ is {\it nilpotent} if $\mathcal{J}(x)$ is nilpotent for any tangent vector $x$;
this does not imply $(M,g)$ is flat in the pseudo-Riemannian setting.

Even
in signature $(2,2)$, the situation is far from clear although much progress has been made
recently by Calvi\~{n}o-Louzao et al \cite{C-GR-VL-07} in examining these questions and
similar questions related to the skew-symmetric curvature operator and 
by D\'{\i}az-Ramos et al \cite{DGV1} in examining non-diagonalize Jacobi operators. 
Derdzinski \cite{D09} has examined
questions concerning type III Jordan-Osserman metrics raised by Diaz-Ramos et al \cite{DGV06}.
Walker geometry is intimately related with many questions in mathematical physics.
 Chaichi et. al. \cite{CGM05} have studied conditions for a Walker metric to be Einstein, Osserman, or
locally conformally flat and obtained thereby exact solutions to the Einstein equations for a
restricted Walker manifold. Chudecki and Prazanowski \cite{CP08,CP08a} 
examined Osserman metrics in terms of $2$-spinors
and provided some new results in HH-geometry using the close relation between weak HH-spaces and Walker
and Osserman spaces using results of \cite{DGV06}.

\subsection{Affine Osserman manifolds} Let $\nabla$ be a torsion free connection on
a smooth manifold $M$; the pair $(M,\nabla)$ is said to be an {\it affine manifold}.
The first work on Osserman geometry in the affine setting
is due to Garc\'{\i}a-R\'{\i}o et al \cite{GKVV}. One has
$\mathcal{J}(\lambda x)=\lambda^2\mathcal{J}(x)$ for $\lambda\in\mathbb{R}$. This
rescaling must be taken into effect. If $T$ is a linear map of a finite dimensional real
vector space, let $\Spec\{T\}\subset\mathbb{C}$ be the {\it spectrum} of $T$;
this is the set of roots of the characteristic polynomial $P_T(\lambda):=\det(T-\lambda\Id)$. 
One says that an affine manifold $(M,\nabla)$
is {\it affine Osserman} if $\operatorname{Spec}(\mathcal{J}(x))=\{0\}$ for any 
tangent vector $x$; i.e. $\mathcal{J}(x)$ is nilpotent. This notion clearly is invariant under
rescaling and there are many examples. One has, for example, the following
result of  Garc\'{\i}a-R\'{\i}o et al \cite{GGNV13}:
\begin{theorem} Define a torsion free connection on $\mathbb{R}^m$ by setting
$$\nabla_{\partial_{x_i}}\partial_{x_j}=\sum_{k>\max(i,j)}\Gamma_{ij}{}^k(x_1,...,x_{k-1})\partial_{x_k}
\text{ for }\Gamma_{ij}{}^k=\Gamma_{ji}{}^k\,.$$
Then $(M,\nabla)$ is affine Osserman.
\end{theorem}

Such examples are important in neutral signature Osserman geometry. Let $(M,\nabla)$ be
an affine manifold. Let $(x^1,...,x^m)$ be local coordinates on $M$. If $\omega\in T^*M$, expand
$\omega=\sum_iy_idx^i$ to define the dual fiber coordinates $(y_1,...,y_m)$ and thereby obtain
{\it canonical local coordinates} $(x^1,...,x^m,y_1,...,y_m)$ on $T^*M$. Let $\Phi=\Phi_{ij}dx^i\circ dx^j$ be a
smooth symmetric $2$-tensor on $M$. The {\it deformed Riemannian extension} $g_{\nabla,\Phi}$
is the metric of neutral signature $(\bar m,\bar m)$ on the cotangent bundle $T^*M$ given locally by
\begin{eqnarray*}
&&g_{\nabla,\phi}(\partial_{x_i},\partial_{x_j})=-2y_k\Gamma_{ij}{}^k(x)+\Phi_{ij}(x),\\
&&g_{\nabla,\phi}(\partial_{x_i},\partial_{y^j})=\delta_i^j,\quad
g_{\nabla,\phi}(\partial_{y^i},\partial_{y^j})=0\,.
\end{eqnarray*}
This is invariantly defined; we refer to Calvino-Louzao et al \cite{CGGV09} for further details. 
One has:
\begin{theorem}\label{thm-1.2}
Let $(M,\nabla)$ be an affine Osserman manifold and let $\Phi$ be a smooth 
symmetric $2$-tensor on $M$. Then the deformed Riemannian extension $(T^*M,g_{\nabla,\Phi})$
is a pseudo-Riemannian nilpotent Osserman manifold of neutral signature.
\end{theorem}

It is possible to modify this construction to produce Osserman metrics with non-nilpotent Jacobi
operators of neutral signature on $T^*M$ Calvino-Louzao et al \cite{X19}. One defines the {\it modified Riemannian
extension} by setting:
\begin{eqnarray*}
&&g_{\nabla,1}(\partial_{x_i},\partial_{x_j})=-2y_k\Gamma_{ij}{}^k(x)+y_iy_j,\\
&&g_{\nabla,1}(\partial_{x_i},\partial_{y^j})=\delta_i^j,\quad
g_{\nabla,1}(\partial_{y^i},\partial_{y^j})=0\,.
\end{eqnarray*}
Again, this is invariantly defined. One has:
\begin{theorem}\label{thm-1.3}
Let $(M,\nabla)$ be an affine Osserman manifold. Then
the modified Riemannian extension
$(T^*M,g_{\nabla,1})$ is a pseudo-Riemannian Osserman manifold of neutral signature
so that if $\xi_\pm\in S^\pm(T^*M,g_{\nabla,1})$, then
$\Spec\{\mathcal{J}(\xi_\pm)\}=\pm(0,1,\frac14)$ with multiplicities
$(1,1,2m-2)$, respectively.
\end{theorem}

Note that the structures can be chosen so that Jacobi operators for the
metrics in Theorem~\ref{thm-1.2} and and in Theorem~\ref{thm-1.3} have non-trivial Jordan normal form.

\subsection{Projectivizing the spectrum}
Since $\mathcal{J}(\lambda\xi)=\lambda^2\mathcal{J}(\xi)$, it is necessary to take this
rescaling into account. This played no role, of course, if we assume that $\Spec\{\mathcal{J}(\xi)\}=\{0\}$
for all $\xi$. But it motivates the following:
\begin{definition}\label{defn-1.4}\rm
\ \begin{enumerate}
\item Let $(M,\nabla)$ be an affine manifold. 
We say $(M,\nabla)$ is an {\it affine projective Osserman} manifold
if given any pair of non-zero tangent vectors $x,y$, there is a real scaling factor $s(x,y)\ne0$ so
$$
\Spec\{\mathcal{J}(y)\}
=s(x,y)\cdot\Spec\{\mathcal{J}(x)\}\ne\{0\}\,.
$$
\item Let $(M,g)$ be a pseudo-Riemannian manifold.
We say $(M,g)$ is {\it spacelike projective Osserman} (resp. {\it timelike projective Osserman})
if given any pair of vectors $x,y$ in $S^+(M,g)$ (resp. $S^-(M,g)$), 
there is a real scaling factor $s(x,y)\ne0$ so
$$\Spec\{\mathcal{J}(y)\}
=s(x,y)\cdot\Spec\{\mathcal{J}(x)\}\ne\{0\}\,.$$
\end{enumerate}
\end{definition}
Although in principle, we allowed $s(x,y)$ to be negative, 
in fact $s(x,y)$ can be chosen to be positive and once this is done, $s$ is smooth. 
We will establish the following result in Section~\ref{sect-2}:
\begin{lemma}\label{lem-1.5}
Let $(M,\nabla)$ be an affine manifold.
Let $\mathcal{O}$ be a connected open subset of $TM$. Suppose
there exists $s(x,y)$ so that $\Spec\{\mathcal{J}x)\}=s(x,y)\Spec\{\mathcal{J}(y)\}\ne\{0\}$
 for all $x,y\in\mathcal{O}$. 
Then:
\begin{enumerate}
\item $\Tr\{\mathcal{J}(x)^k\}=s(x,y)^k\Tr\{\mathcal{J}(y)^k\}$ for any $x,y\in\mathcal{O}$ and any $k$.
\item $\Spec\{\mathcal{J}(x)\}=|s(x,y)|\Spec\{\mathcal{J}(x)\}\ne\{0\}$ for all $x,y\in\mathcal{O}$.
\item There exists $k$ so that
$$|s(x,y)|=\left\{\frac{\Tr\{\mathcal{J}(x)^k\}}{\Tr\{\mathcal{J}(y)^k\}}\right\}^{1/k}
\text{ for any }x,y\in\mathcal{O}\,.$$
\item The function $|s(x,y)|$ is smooth on $\mathcal{O}\times\mathcal{O}$.
\end{enumerate}
\end{lemma}

Since $S^\pm(M,g)$ has at most two components and since $\mathcal{J}(-\xi)=\mathcal{J}(\xi)$,
the following result is an immediate consequence of Lemma~\ref{lem-1.5}:

\begin{theorem}\ 
\begin{enumerate}
\item Let $(M,\nabla)$ be an affine projective Osserman manifold. Then the function
$s(x,y)$ of Definition~\ref{defn-1.4}~(1) can be taken to be positive and smooth.
\item Let $(M,g)$ be a
spacelike projective Osserman (resp. timelike projective Osserman) manifold. 
Then the function $s(x,y)$
in Definition~\ref{defn-1.4}~(2) can be taken to be positive and smooth.
\end{enumerate}
\end{theorem}

The notions of timelike Osserman and spacelike Osserman are equivalent (see Garc\'\i a-R\'\i o et al \cite{GKVa}).
This
is not the situation in the setting at hand as we shall show in Section~\ref{sect-3}:

\begin{theorem}\label{thm-1.7}
Let $p>0$ and $q>0$. There exists a pseudo-Riemannian manifold $(M,g)$
of signature $(p,q)$
which is spacelike projective Osserman but not timelike projective
Osserman. Similarly, there exists a pseudo-Riemannian manifold $(\tilde M,\tilde g)$
of signature $(p,q)$
which is timelike projective Osserman but not spacelike projective Osserman.
\end{theorem}

In Section~\ref{sect-4}, we will generalize Theorem~\ref{thm-1.2} to the projective setting:

\begin{theorem}\label{thm-1.8}
Let $\Phi$ be a symmetric $2$ tensor on an affine manifold
$(M,\nabla)$.
\item The following assertions are equivalent:
\begin{enumerate}
\item $(M,\nabla)$ is an affine projective Osserman manifold.
\item $(T^*M,g_{\nabla,\Phi})$ is a spacelike projective Osserman manifold.
\item $(T^*M,g_{\nabla,\Phi})$ is timelike projective Osserman manifold.
\end{enumerate}
\end{theorem}

Let $\rho(x,y):=\Tr\{z\rightarrow R(z,x)y\}$ be the Ricci tensor. This tensor need no longer be
symmetric so we let $\rho_s(x,y):=\frac12\{\rho(x,y)+\rho(y,x)\}$ be the symmetric part of this
tensor. Any Riemannian Osserman manifold is necessarily an affine projective Osserman manifold; the fact that
$(M,g)$ is Riemannian is crucial here since if $\mathcal{J}(\xi)$ is nilpotent if $\xi$ is null
for a higher signature pseudo-RiemannianOsserman manifold.
Consequently if $(M,g)$ is a rank $1$ symmetric space, then
$(M,g)$ is an affine projective Osserman manifold. 
If $m=2$ and if $0\ne x$, let $\{0,\lambda(x)\}$ be the eigenvalues of $\mathcal{J}(x)$ where
each eigenvalue is repeated according to its multiplicity.
Then $\rho(x,x)=\rho_s(x,x)=\Tr\{\mathcal{J}(x)\}=\lambda(x)$. The following result is now immediate
and provides examples to which Theorem~\ref{thm-1.8} applies:

\begin{theorem}\label{thm-1.9}
\ \begin{enumerate}
\item Any rank $1$-symmetric space is an affine projective Osserman manifold where we let $\nabla$ be the
Levi-Civita connection.
\item If $m=2$ and if $(M,\nabla)$ is an affine manifold, then
$(M,\nabla)$ is an affine projective Osserman manifold if and only if $\rho_s(x,x)\ne0$ for
all $x$, i.e. $\rho_s$ is definite.
\end{enumerate}
\end{theorem}

\subsection{The algebraic context}
Let $V$ be a real vector space of dimension $m$ and let $A\in\operatorname{End}(V)\otimes V^*$.
We say that $(V,A)$ is an {\it affine curvature model} if $A$ has the symmetries of the curvature operator
of an affine connection for all $x,y,z\in V$:
\begin{eqnarray*}
&&A(x,y)z=-A(y,x)z,\\
&&A(x,y)z+A(y,z)x+A(z,x)y=0\,.
\end{eqnarray*}
The first symmetry is the $\mathbb{Z}_2$ anti-symmetry and the second symmetry is the first Bianchi identity.
If $(M,\nabla)$ is an affine manifold, then $(T_PM,R_P)$ is an affine curvature model for any $P\in M$.
Conversely, given an affine curvature model $(V,A)$, then there exists a complete affine manifold $(M,\nabla)$
and a point $P$ of $M$ so that $(V,A)$ is isomorphic to $(T_PM,R_P)$, i.e. 
every affine curvature model can be geometrically realized by a complete affine manifold 
(see Y. Euh et al \cite{BGN12}).

Let $(V,A)$ be an affine curvature model. The associated {\it Jacobi operator} is given by setting
$\mathcal{J}(v)w:=A(w,v)v$. One says that $(V,A)$ is an {\it affine projective Osserman} curvature model
if
$\Spec\{\mathcal{J}(v)\}=s(v,w)\Spec\{\mathcal{J}(w)\}\ne\{0\}$ for $0\ne v,w\in V$.
In Section~\ref{sect-5}, we will prove the following result which has an immediate application to the
geometric setting:

\begin{theorem}\label{thm-1.10}
Let $(V,A)$ be a an affine projective Osserman curvature model of odd dimension $m$.
If $0\ne v\in V$, then $\Spec\{\mathcal{J}(v)\}=\{0,\lambda(v)\}$ where $\lambda(v)$ is
a smooth real valued function on $V-\{0\}$ which never vanishes. The eigenvalue $0$ appears with
multiplicity $1$ and the eigenvalue $\lambda(v)$ appears
with multiplicity $m-1$. 
In this situation $\rho(v,v)=(m-1)\lambda(v)$ so the symmetric Ricci tensor $\rho_s$ defines
a non-degenerate definite inner product on $V$.\end{theorem}

In Section~\ref{sect-6}, we will prove the following result:

\begin{theorem}\label{thm-1.11}
Let $\mathfrak{M}_\varepsilon:=(\mathbb{R}^m,A)$ where the non-zero components of $A$ are determined by:
$$A_{ijj}{}^i=1\text{ for }1\le i\ne j\le m\text{ and }A_{122}{}^2=A_{121}{}^1=-\varepsilon\,.$$
\begin{enumerate} 
\item $\mathfrak{M}_\varepsilon$ is an affine projective Osserman model for any $\epsilon$.
\item $\mathfrak{M}_\varepsilon$ is geometrically realizable by an affine projective Osserman manifold.
\end{enumerate}
\end{theorem}

\begin{remark}\rm
The Ricci tensor of the model in Theorem~\ref{thm-1.11} is given by
$$\rho(e_i,e_j)=\left\{\begin{array}{rl}
\varepsilon&\text{ if }i=1,\ j=2\\
-\varepsilon&\text{ if }i=2,\ j=1\\
m-1&\text{ if }i=j\\
0&\text{ otherwise }\end{array}\right\}\,.$$
If $\varepsilon\ne0$, then $\rho_s$ is not symmetric and $A$ is not a Riemannian
algebraic curvature operator and, in particular, is not the curvature operator of constant sectional
curvature $+1$.
\end{remark}

The tensor of Theorem~\ref{thm-1.11} is a perturbation of the curvature tensor of constant sectional
curvature $+1$. In Section~\ref{sect-7}, we present two algebraic examples which are perturbations
of the Fubini-Study metric on complex projective space and on quaternionic projective space, respectively,
and which are affine projective Osserman models.

\section{The proof of Lemma~\ref{lem-1.5}}\label{sect-2}

Let $(M,\nabla)$ be an affine manifold and let $\mathcal{O}$ be an open connected subset of $TM$.
Assume $\Spec\{\mathcal{J}(x)\}=s(x,y)\Spec\{\mathcal{J}(y)\}\ne\{0\}$ for all $x,y\in\mathcal{O}$.
 Let $\sigma(t)$ be a path in $\mathcal{O}$. Since the number of eigenvalues of 
$\Spec\{\mathcal{J}(\sigma(t))\}$ is independent of $t$, eigenvalues
do not coalesce or bifurcate and consequently the eigenvalue multiplicities are constant as well along
$\sigma$. Thus
\begin{equation}\label{eqn-2.a}
\Tr\{\mathcal{J}(t)^k\}=s(\sigma(0),\sigma(t))^k\Tr\{\mathcal{J}(0)^k\}\text{ for any }k\,.
\end{equation}
Since $\mathcal{J}(\sigma(0))$ is not nilpotent, 
$\Tr\{\mathcal{J}(\sigma(0))^k\}\ne0$ for some $k$.
Fix such a $k$. Since $\mathcal{O}$ is connected, Equation~(\ref{eqn-2.a}) implies that
$\Tr\{\mathcal{J}(x)^k\}\ne0$ for any $x\in\mathcal{O}$ and that $s(\sigma(0),\sigma(t))^k$
is smooth.  If $k$ is odd, 
since $s(\sigma(0),\sigma(0))=1$ and $s(\sigma(0),\sigma(t))\ne0$, we have 
$s(\sigma(0),\sigma(t))>0$. Since the endpoints were arbitrary, $s(x,y)>0$ for all $(x,y)$ and the Lemma follows.

On the other hand, if
$\Tr\{\mathcal{J}(\sigma(0))^k\}=0$ for all odd $k$,
then $\Spec\{\mathcal{J}(\sigma(0))\}$ is symmetric about the origin
and we may assume $s(\sigma(0),\sigma(t))$ is positive.
Again, we can take the $k^{\operatorname{th}}$ root
to establish Lemma~\ref{lem-1.5}.\hfill\qed

\section{The proof of Theorem~\ref{thm-1.7}}\label{sect-3}
Let $p>0$ and let $q>0$ be given.
Let $(S^q,g_q)$ denote the sphere in $\mathbb{R}^{q+1}$
with the standard metric of constant sectional curvature $+1$. Let $(\mathbb{R}^p,g_p)$
denote $\mathbb{R}^p$ with a flat negative definite metric. Let
$M=(\mathbb{R}^p\times S^q,g_p\oplus g_q)$; this metric has signature $(p,q)$.
If $\xi=(\xi_p,\xi_q)\in TM$, then $\mathcal{J}(\xi)=0\oplus\mathcal{J}(\xi_q)$. If $\xi$ is spacelike,
then $\xi_q\ne0$ and $\Spec\{\mathcal{J}(\xi)\}=\Spec\{\mathcal{J}(\xi_q)\}=\{0,|\xi_q|^2\}$
 and thus $(M,g)$ is a spacelike projective Osserman manifold; $0$ is an eigenvalue
 of multiplicity $p$. On the other hand, if $\xi_q=0$
 and $\xi_p\ne0$, then $\xi$ is timelike and $\Spec\{\mathcal{J}(\xi)\}=\{0\}$ so $(M,g)$
 is not a timelike projective Osserman manifold. This proves the first assertion of Theorem~\ref{thm-1.7}; the
 second follows similarly.\hfill\qed
 
\section{The proof of Theorem~\ref{thm-1.8}}\label{sect-4}
Let $\sigma$ be the canonical projection from $T^*M$ to $M$. 
Let $\xi\in T(T^*M)$ and let $a=\sigma_*\xi\in TM$. Relative to
the canonical frame $(\partial_{x_1},...,\partial_{x_m},\partial_{y^1},...,\partial_{y^m})$ for $T(T^*M)$,
one has  (see Garc\'{\i}a-R\'{\i}o et al \cite{GKVV}) that:
$$
\mathcal{J}_{g_{\nabla,\Phi}}(\xi) = \left(
\begin{array}{ll}
\mathcal{J}_\nabla({a}) & 0\\\ast & {}^t\mathcal{J}_\nabla({a})
\end{array}\right)
$$
where $\ast$ is some linear map from 
$\operatorname{Span}\{\partial_{x_i}\}$ to $\Span\{\partial_{y^k}\}$. Consequently
$$
\Spec\{\mathcal{J}_{g_{\nabla,\Phi}}(\xi)\}=\Spec\{\mathcal{J}_\nabla(a)\}\,.
$$
If $\xi_\pm\in S^\pm(T^*M,g_{\nabla,\Phi})$,
then $a:=\sigma_*\xi_\pm\ne0$. The implication (1) $\Rightarrow$ (2) and
the implication (1) $\Rightarrow$ (3)
of Theorem~\ref{thm-1.8} now follow. Conversely, suppose that Assertion~(2)
holds or that Assertion~(3) holds.
Let $a\ne0$. Choose
$\xi\in S^\pm(T^*M,g_{\nabla,\Phi})$ so that $\sigma_*(\xi)=ta$ for some $t\ne0$; 
The implications (2) $\Rightarrow$ (1) and (3) $\Rightarrow$ (1) now follow.
\hfill\qed

\section{The proof of Theorem~\ref{thm-1.10}}\label{sect-5}

Let $(V,A)$ be an affine projective Osserman curvature model.
Fix a basepoint $0\ne x\in V$ and let
$\Spec\{\mathcal{J}(x)\}=\{0,\lambda_1,...\}$; by hypothesis $\Spec\{\mathcal{J}(x)\}\ne\{0\}$.
If $0\ne y\in V$, $\Spec\{\mathcal{J}(y)\}=\{0,s(y,x)\lambda_1,...\}$.
Let
$$V_1(y)=\left\{\begin{array}{l}
\ker\{(\mathcal{J}(y)-s(y,x)\lambda_1)^m\}\text{ if }\lambda_1\in\mathbb{R}\\
\ker\{(\mathcal{J}(y)-s(y,x)\lambda_1)^m(\mathcal{J}(y)-s(y,x)\bar\lambda_1)^m\}\text{ otherwise}
\end{array}\right\}$$
be the generalized eigenspace corresponding to $\lambda_1$ if $\lambda_1$ is real and to 
$\{\lambda_1,\bar\lambda_1\}$
otherwise. 
As noted previously, the eigenvalue multiplicities are constant. Thus
these generalized eigenspaces have constant dimension and vary smoothly with $y$.
Put an auxiliary inner product $\langle\cdot,\cdot\rangle$ on $V$ and let
$S^{m-1}$ be the unit sphere of $(V,\langle\cdot,\cdot\rangle)$. Let $y\in S^{m-1}$.
Since $\mathcal{J}(y)y=0$,
$\mathcal{J}(y)$ induces an endomorphism of the quotient space
$V/y\cdot\mathbb{R}$ which we may identify with
$T_yS^{m-1}$. Since $m-1$ is even, $S^{m-1}$ has no non-trivial sub-bundles. Since $\{0\}\ne V_1$
is a sub-bundle of $TS^{m-1}$, we conclude $V_1=TS^{m-1}$ for $y\in S^{m-1}$.
This implies that $0$ is an eigenvalue
of multiplicity $1$ and that
$$\Spec\{\mathcal{J}(y)\}=\left\{\begin{array}{l}
\{0,s(y,x)\lambda_1\}\text{ if }\lambda\in\mathbb{R}\\
\{0,s(y,x)\lambda_1,s(y,x)\bar\lambda_1\}\text{ otherwise}\end{array}\right\}\,.$$
This completes the proof if $\lambda_1$ is real. Thus we suppose $\lambda_1$ is complex and
argue for a contradiction. We complexify and decompose
$$T_y(S^{m-1})\otimes_{\mathbb{R}}\mathbb{C}
=W_{s(y,x)\lambda_1}\oplus W_{s(y,x)\bar\lambda_1}$$
into the generalized eigenbundles corresponding to $\lambda$ and $\bar\lambda$ where
$$
W_\mu(y):=\{\xi\in T_yS^{m-1}\otimes_{\mathbb{R}}\mathbb{C}:(\mathcal{J}(y)-\mu)^m\xi=0\}\,.
$$
Since 
$\mathcal{J}(-y)=\mathcal{J}(y)$, we obtain a corresponding decomposition of the tangent bundle
of projective space 
$$T(\mathbb{RP}^{m-1})\otimes_{\mathbb{R}}\mathbb{C}=W_\lambda\oplus W_{\bar\lambda}\,.$$
Since $W_\lambda=\bar W_{\bar\lambda}$, the first Chern class vanishes:
$$0=c_1(T(\mathbb{RP}^{m-1})\otimes_{\mathbb{R}}\mathbb{C})
\in H^2(\mathbb{RP}^{m-1};\mathbb{Z}_2)=\mathbb{Z}_2\,.$$
On the other hand, $\mathbb{RP}^{m-1}$ is not orientable since $m-1$ is even. Thus
$w_1(T(\mathbb{RP}^{m-1}))$ generates $H^1(\mathbb{RP}^{m-1};\mathbb{Z}_2)=\mathbb{Z}_2$.
Since the generator of the first cohomology group
$H^1(\mathbb{RP}^{m-1};\mathbb{Z}_2)$ squares to the generator of the second cohomology group
$H^2(\mathbb{RP}^{m-1};\mathbb{Z}_2)$, this implies
$$0\ne w_1^2(T(\mathbb{RP}^{m-1}))\in H^2(\mathbb{RP}^{m-1};\mathbb{Z}_2)=\mathbb{Z}_2\,.$$
This is a contradiction since
$$
w_1^2(T(\mathbb{RP}^{m-1}))=c_1(T(\mathbb{RP}^{m-1})\otimes_{\mathbb{R}}\mathbb{C})\,.
$$
This contradiction completes the proof.\hfill\qed

\section{The proof of Theorem~\ref{thm-1.11}}\label{sect-6}
If $m=2$, then Theorem~\ref{thm-1.11} follows from Theorem~\ref{thm-1.9}~(2) so we shall assume that $m\ge3$.
We have defined $\mathfrak{M}_\varepsilon:=(\mathbb{R}^m,A)$, where the non-zero components of $A$
are determined by:
$$A_{ijj}{}^i=1\text{ for }1\le i\ne j\le m\quad\text{ and }\quad A_{122}{}^2=A_{121}{}^1=-\varepsilon.$$
If $(\cdot,\cdot)$ is the usual Euclidean inner product on $\mathbb{R}^m$, then $\rho_s=(m-1)(\cdot,\cdot)$.
We let $G:=SO(2)\times SO(m-2)$ act on $\mathbb{R}^m$. We lower indices and regard
$A\in\otimes^4V^*$:
$$
A=-\varepsilon(e^1\wedge e^2)\otimes(e^1\otimes e^1+e^2\otimes e^2)
+\sum_{i<j}(e^i\wedge e^j)\otimes(e^j\wedge e^i)\,.
$$
Consequently $A$ is invariant under the action of $G$ so $\Spec\{\mathcal{J}(x)\}=\Spec\{\mathcal{J}(gx)\}$
for all $g\in G$.
Let $x=a_1e_1+...+a_me_m$ belong to $S^{m-1}$.
We may use the action of $SO(2)$ to ensure that $a_2=0$ and we may use the
the action of $SO(m-2)$ to ensure that
$a_i=0$ for $i>3$ in examining $\Spec\{\mathcal{J}(x)\}$. Thus we may assume that $x=\cos\theta e_1+\sin\theta e_3$ so
\begin{eqnarray*}
&&\mathcal{J}(x)e_i=e_i\text{ for }i\ge 4,\\
&&\mathcal{J}(x)(\cos\theta e_1+\sin\theta e_3)=0,\\
&&\mathcal{J}(x)(-\sin\theta e_1+\cos\theta e_3)=-\sin\theta e_1+\cos\theta e_3\,.
\end{eqnarray*}
Thus $0$ is an eigenvalue of multiplicity at least $1$ and $+1$ is a eigenvalue of multiplicity at least $m-2$.
Since $\Tr\{\mathcal{J}(x)\}=\rho(x,x)=(m-1)$, we conclude that $+1$ is an eigenvalue of multiplicity $m-1$.
Consequently, $\mathfrak{M}_\varepsilon$ is an affine projective Osserman curvature model
 for any $\varepsilon$.

Define a torsion free
connection $\nabla$ on $\mathbb{R}^m$ by setting:
$$
\Gamma_{mm}{}^m=2,\ \Gamma_{im}{}^i=\Gamma_{mi}{}^i=\Gamma_{ii}{}^m=1\text{ for }i<m;
\ \Gamma_{11}{}^1=-\Gamma_{22}{}^2=\varepsilon(x_1+x_2)\,.
$$
We have
$R_{ijk}{}^l=\partial_{x_i}\Gamma_{jk}{}^l-\partial_j\Gamma_{ik}{}^l+\Gamma_{in}{}^l\Gamma_{jk}{}^n-
\Gamma_{jn}{}^l\Gamma_{ik}{}^n$.
There are no terms in $\varepsilon^2$ and the only terms in $\varepsilon$ which are quadratic
in the Christoffel symbols are
\begin{eqnarray*}
0&=&\Gamma_{m1}{}^1\Gamma_{11}{}^1-\Gamma_{11}{}^1\Gamma_{m1}{}^1=0,\\
0&=&\Gamma_{m2}{}^2\Gamma_{22}{}^2-\Gamma_{22}{}^2\Gamma_{m2}{}^2=0\,.
\end{eqnarray*}
Consequently, the quadratic terms give rise to:
\begin{eqnarray*}
&&R_{imm}{}^i=\Gamma_{im}{}^i\Gamma_{mm}{}^m-\Gamma_{mi}{}^i\Gamma_{im}{}^i
=2-1\text{ for }i<m,\\
&&R_{mii}{}^m=\Gamma_{mm}{}^m\Gamma_{ii}{}^m-\Gamma_{ii}{}^m\Gamma_{mi}{}^i
=2-1\text{ for }i<m,\\
&&R_{ijj}{}^i=\Gamma_{im}{}^i\Gamma_{jj}{}^m=1\text{ for }i\ne j<m\,.
\end{eqnarray*}
We complete the proof by examining the terms involving the derivatives of $\Gamma$ and verifying:
\medbreak\hfill $R_{122}{}^2=\partial_{x_1}\Gamma_{22}{}^2=-\varepsilon$ and 
$R_{211}{}^1=\partial_{x_2}\Gamma_{11}{}^1=\varepsilon$.\hfill\qed

\begin{remark}\label{rmk-6.1}\rm 
Suppose $\varepsilon\ne0$. If $\theta=\frac\pi2$, then $x=e_3$ and
$\mathcal{J}(x)$ is diagonal. If $\theta=\frac\pi4$, then
$x=\frac1{\sqrt2}(e_1+e_3)$ and:
\begin{eqnarray*}
&&\mathcal{J}(x)(e_1+e_3)=0,\quad
\mathcal{J}(x)(e_1-e_3)=e_1-e_3,\quad
\mathcal{J}(x)e_2=\textstyle\frac12\varepsilon e_1+e_2,\\
&&\mathcal{J}(x)\{e_2+\textstyle\frac14\varepsilon(e_1+e_3)\}=\frac12\varepsilon e_1+e_2
=e_2+\textstyle\frac14\varepsilon(e_1+e_3)+\textstyle\frac14\varepsilon(e_1-e_3)\,.
\end{eqnarray*}
Thus the space $\Span\{u_1:=e_2+\textstyle\frac14\varepsilon(e_1+e_3),v_2=\frac14\varepsilon(e_1-e_3)\}$
is invariant under the action of $\mathcal{J}(x)$ and we have $\mathcal{J}(x)v_2=v_2$ and
$\mathcal{J}(x)v_1=v_1+v_2$. Consequently, we have non-trivial Jordan normal form in this instance.
\end{remark}

\begin{remark}\rm
Suppose that $\varepsilon=0$. 
Since the Christoffel symbols are constant, the group of translations acts on 
transitively on $\mathcal{M}$ by affine isomorphisms; thus $\mathcal{M}$ is affine
homogeneous\index{homogeneous}. However, if we set $\sigma(t)=(0,...,0,x(t))$, then the
geodesic equation\index{geodesic equation}
becomes $\ddot x+\dot x\dot x=0$
which blows up in finite time for suitable initial conditions. Thus
$(\mathbb{R}^m,\nabla)$ is geodesically incomplete\index{geodesically incomplete}. Finally, 
we compute:
\begin{eqnarray*}
&&\nabla R(\partial_m,\partial_1,\partial_1;\partial_1)\\
&=&\nabla_{\partial_m}R(\partial_m,\partial_1)\partial_1
-R(\nabla_{\partial_m}\partial_m,\partial_1)\partial_1-R(\partial_m,\nabla_{\partial_m}\partial_1)\partial_1
-R(\partial_m,\partial_1)\nabla_{\partial_m}\partial_1\\
&=&(2-2-2)\partial_m\ne0\,.
\end{eqnarray*}
Consequently, $\nabla\mathcal{R}\ne0$. Thus these manifolds
are not locally symmetric\index{symmetric}. This shows that the affine manifolds $\mathcal{M}_0$
are not affinely equivalent to the standard affine structure on the sphere $S^m$.

If $\varepsilon\ne0$, then there is a translation group of rank $m-1$ which acts on $(M,\nabla)$
preserving the structures. Furthermore, this manifold is affine curvature homogeneous. However,
we have additional entries in $\nabla R$:
$$
\nabla R(\partial_2,\partial_1,\partial_1;\partial_1)=-2\Gamma_{11}{}^1\partial_2,\text{ and }
\nabla R(\partial_1,\partial_2,\partial_2;\partial_2)=-2\Gamma_{22}{}^2\partial_1\,.
$$
Since $\Gamma_{11}{}^1$ and $\Gamma_{22}{}^2$ vanishe if and only if $x_1+x_2=0$,
$(M,\nabla)$ is not 1-affine curvature homogeneous and has
affine cohomogeneity 1.\end{remark}

\section{Two algebraic examples}\label{sect-7}

In Section~\ref{sect-6}, we considered a model based on the tensor of constant sectional curvature 1. 
In this section, we examine examples which are related to the curvature operators of complex and
projective space. These examples have
non-symmetric Ricci tensors and non-trivial Jordan
normal form. We do not know if any of the examples in this section can be realized geometrically.

\subsection{A complex example}
\rm Let $m=2\bar m$ be even, let $(\cdot,\cdot)$ be the usual positive definite inner product on $\mathbb{R}^{m}$
for $m$ even, and
let $J$ be a Hermitian complex structure; this means that
$$J^*(\cdot,\cdot)=(\cdot,\cdot)\text{ and }J^2=-\operatorname{Id}\,.$$
We can choose an orthonormal basis $\{e_1,...,e_m\}$ for $\mathbb{R}^m$ so that if $1\le j\le\bar m$, then:
$$Je_i=\left\{\begin{array}{l}e_{2j}\text{ if }i=2j-1\\
-e_{2j-1}\text{ if }i=2j\end{array}\right\}\,.
$$
Define algebraic affine curvature operators by setting:
\begin{eqnarray*}
&&A_0(x,y)z:=(y,z)x-(x,z)y,\\
&&A_J(x,y)z:=(Jy,z)Jx-(Jx,z)Jy-2(Jx,y)Jz,\\
&&\mathcal{E}(e_1,e_2)e_1=-e_1\text{ and }\mathcal{E}(e_2,e_1)e_2=e_2\,.
\end{eqnarray*}
The tensor $A_0+A_J$ is the curvature operator of the Fubini-Study metric
on complex projective space $\mathbb{CP}^{\bar m}$. If $(x,x)=1$, then
$$
\mathcal{J}_{\lambda_0A_0+\lambda_1A_J}(x)\cdot{ y}=\left\{\begin{array}{rll}
0&\text{if}&y=x\\
(\lambda_0+3\lambda_1)y&\text{if}&y=Jx\\
\lambda_0y&\text{if}&y\perp\{x,x\}\end{array}\right\}\,.
$$
Thus $\lambda_0A_0+\lambda_1A_J$ is an affine projective Osserman curvature model.

\begin{lemma} 
Let $\mathfrak{M}_\varepsilon:=(\mathbb{R}^m,\lambda_0A_0+\lambda_1A_J+\varepsilon\mathcal{E})$.
The eigenvalues of $\mathcal{J}(x)$ for $(x,x)=1$ are $\{0,\lambda_0+3\lambda_1,\lambda_0,...,\lambda_0\}$
where each eigenvalue is repeated according to multiplicity. Thus $\mathfrak{M}_\varepsilon$ is
an affine projective Osserman curvature model.
\end{lemma}

\begin{proof} The tensor $A_0$ is invariant under the action of the full orthogonal group $O(m)$, 
the tensor $A_J$ is invariant under the action of the unitary group $U(\bar m)$,
and the tensor $\mathcal{E}$ is invariant under
the action of the group
$$G:=U(1)\times U(\bar m-1)\subset U(\bar m)\subset O(m)\,.$$
We suppose $x_1\in\mathbb{R}^{m}$
satisfies $(x_1,x_1)=1$. We use the action of $G$ to assume 
$x_1=\cos\theta e_1+\sin\theta e_3$ when studying $\mathcal{J}(x)$.
Then $\mathcal{J}(x)=\operatorname{Id}$ on $\Span\{e_i\}_{i\ge5}$
so this space plays no role and we may assume $m=4$. Let 
$$\begin{array}{ll}
   x_1:=\cos\theta e_1+\sin\theta e_3,&x_2:=-\sin\theta e_1+\cos\theta e_3,\\
   x_3:=Jx_1=\cos\theta e_2+\sin\theta e_4,&
   x_4:=Jx_2=-\sin\theta e_2+\cos\theta e_4.\end{array}$$
Let $\star$ be a coefficient which we do not need to specify. We then have
\begin{eqnarray*}
&&\mathcal{J}(x_1)x_1=0,\\
&&\mathcal{J}(x_1)x_2=\lambda_0x_2,\\
&&\mathcal{J}(x_1)x_3=(\lambda_0+3\lambda_1)x_3+\star e_1
=\star x_1+\star x_2+(\lambda_0+3\lambda_1)x_3,\\
&&\mathcal{J}(x_1)x_4=\lambda_0x_4+\star e_1=\star x_1+\star x_2+\lambda_0x_4\,.
\end{eqnarray*}
The matrix of $\mathcal{J}(x_1)$ on this 4-dimensional subspace is therefore given by
$$\mathcal{J}(x_1)=\left(\begin{array}{llll}
0&0&\star&\star\\
0&\lambda_0&\star&\star\\
0&0&\lambda_0+3\lambda_1&0\\
0&0&0&\lambda_0\end{array}\right)\,.$$
The Lemma now follows.\end{proof}

\begin{remark}\label{rmk-7.2}
\rm If we take $\theta=\frac\pi2$, then $x_1=e_3$ and $\mathcal{J}(x_1)$ is diagonal. If we take
$\theta=\frac\pi4$, then $x_1=(e_1+e_3)/\sqrt{2}$ and the same argument given in Remark~\ref{rmk-6.1}
shows $\mathcal{J}(x_1)$ has non-trivial Jordan normal form:
\begin{eqnarray*}
&&\mathcal{J}(x)(e_1+e_3)=0,\quad
\mathcal{J}(x)(e_1-e_3)=e_1-e_3,\\
&&\mathcal{J}(x)(e_2-e_4)=\textstyle\frac12\varepsilon e_1+(e_2-e_4),\\
&&\mathcal{J}(x)\{e_2-e_4+\textstyle\frac14\varepsilon(e_1+e_3)\}=\frac12\varepsilon e_1+e_2-e_4\\
&&\qquad=e_2-e_4+\textstyle\frac14\varepsilon(e_1+e_3)+\textstyle\frac14\varepsilon(e_1-e_3)\,.
\end{eqnarray*}
and, again, the Jordan normal form is non-trivial if $\varepsilon\ne0$.
\end{remark}

\subsection{A quaternion example}
Let $m=4k$ and let $\{J_1,J_2,J_3\}$ give $\mathbb{R}^{4k}$ an orthogonal quaternion structure, i.e.
$$
(J_ix,J_ix)=(x,x),\quad
J_iJ_j+J_jJ_i=-2\delta_{ij}\operatorname{Id},\text{ and }J_1J_2=J_3\,.
$$
Let
$$\mathcal{E}:=-\left\{(e^1\wedge e^2)\otimes(e^1\otimes e^1+e^2\otimes e^2)
+(e^3\wedge e^4)(e^3\otimes e^3+e^4\otimes e^4)\right\}\,.$$

\begin{lemma} Let $\mathfrak{M}_\varepsilon:(\mathbb{R}^m,\lambda_0A_0+\lambda_1A_{J_1}+\lambda_2A_{J_2}+\lambda_3A_{J_3}+\varepsilon\mathcal{E})$.
The eigenvalues of $\mathcal{J}(x)$ for $(x,x)=1$ are $\{0,\lambda_0+3\lambda_1,\lambda_0+3\lambda_2,\lambda_0+3\lambda_3,\lambda_0,...,\lambda_0\}$
where each eigenvalue is repeated according to multiplicity. Thus
$\mathfrak{M}_\varepsilon$ is an affine projective Osserman curvature model.
\end{lemma}

\begin{proof} Let $\mathbb{H}:=\Span_{\mathbb{R}}\{1,i,j,k\}$ denote the quaternions.
We take an orthonormal basis $\{e_1^\nu,e_i^\nu,e_j^\nu,e_k^\nu\}$ for $\mathbb{R}^{4k}$ where
$1\le\nu\le k$ so that $e_i^\nu=J_1e_1^\nu$, $e_j^\nu=J_2e_1^\nu$, and $e_k^\nu=J_3e_1^\nu$.
This permits us to
identify $\mathbb{R}^m=\mathbb{H}^k$ with the quaternions where $\{J_1=i,J_2=j,J_3=k\}$ are
the quaternions acting from the left. Let $\operatorname{Sp}(k)$ be
the group of isometries of $\mathbb{R}^m$ which commute $\{J_1,J_2,J_3\}$; this is the set of $k\times k$
orthogonal quaternion matrices acting from the right. The affine algebraic curvature tensor in question
is invariant under the action of $\operatorname{Sp}(1)\times\operatorname{Sp}(k-1)$. 
Consequently, in considering $\Spec\{\mathcal{J}(x)\}$, it suffices to consider the special case
$x=\cos\theta e_1^1+\sin\theta e_1^2$. The remaining variables $e_*^\nu$ for $\nu\ge3$ play
no role and may be ignored. We compute:
\medbreak\qquad
$\mathcal{J}(x)(\cos\theta e_1^1+\sin\theta e_1^2)=0$,
\smallbreak\qquad
$\mathcal{J}(x)(-\sin\theta e_1^1+\cos\theta e_1^2)=\lambda_0(-\sin\theta e_1^1+\cos\theta e_1^2)$,
\smallbreak\qquad
$\mathcal{J}(x)(\cos\theta e_i^1+\sin\theta e_i^2)
             =(\lambda_0+3\lambda_1)(\cos\theta e_i^1+\sin\theta e_i^2)+\star e_1^i$,
\smallbreak\qquad
$\mathcal{J}(x)(-\sin\theta e_i^1+\cos\theta e_i^2)
=\lambda_0(-\sin\theta e_i^1+\cos\theta e_i^2)+\star e_1^i)$,
\smallbreak\qquad
$\mathcal{J}(x)(\cos\theta e_j^1+\sin\theta e_j^2)
=(\lambda_0+3\lambda_2)(\cos\theta e_j^1+\sin\theta e_j^2)$,
\smallbreak\qquad
$\mathcal{J}(x)(-\sin\theta e_j^1+\cos\theta e_j^1)
=\lambda_0(-\sin\theta e_j^2+\cos\theta e_j^2)$,\smallbreak\qquad
$\mathcal{J}(x)(\cos\theta e_k^1+\sin\theta e_k^2)
=(\lambda_0+3\lambda_1)(\cos\theta e_k^1+\sin\theta e_k^2)$,
\smallbreak\qquad
$\mathcal{J}(x)(-\sin\theta e_k^1+\cos\theta e_k^1)=\lambda_0(-\sin\theta e_k^2+\cos\theta e_k^2)$.
\medbreak\noindent The last 4 vectors play no role and the matrix of $\mathcal{J}(x)$ with respect to the first
4 vectors takes the form:
$$\mathcal{J}(x)=\left(\begin{array}{rrrr}
0&0&\star&\star\\
0&\lambda_0&\star&\star\\
0&0&\lambda_0+3\lambda_1&0\\
0&0&0&\lambda_0
\end{array}\right)\,.$$
The desired result now follows.\end{proof}

\subsection{Acknowledgments}
The research of the authors was partially supported by Project
MTM2009-07756 (Spain) and by Project 174012 (Srbija). The paper is dedicated to our friend and colleague
Franki Dillen; may he rest in peace.


\begin{thebibliography}{99}

\bibitem{BM12}  M. Brozos-V\'{a}zquez and E. Merino,
``Equivalence between the Osserman condition and the Raki\'c duality principle in dimension 4",
{\it J. Geom. Phys. \bf 62} (2012), no. 12, 2346--2352.
 
 \bibitem{BGN12} Y. Euh, P. Gilkey, J.H. Park, and K. Sekigawa, ``Transplanting geometric structures",
{\it Differential Geom. Appl.} {\bf 31} (2013) 374--387. 
 \bibitem{BVV09} M. Brozos-V\'{a}zquez,  M. E. V\'{a}zquez-Abal, and R. V\'{a}zquez-Lorenzo, 
 "Conformally Osserman multiply warped product structures in the Riemannian setting",
{\it Differential geometry}, 185--194, World Sci. Publ., Hackensack, NJ, 2009.

\bibitem{CGGV09} E. Calvino-Louzao, E. Garc\'{\i}a--R\'{\i}o, P. Gilkey, and 
R. V\'{a}zquez-Lorenzo,
``The geometry of
modified Riemannian extensions",
{\it Proc. R. Soc. A. \bf 465} (2009), 2023--2040.

\bibitem{X19} E. Calvino-Louzao, E. Garc\'\i a-R\'\i o, P. Gilkey, and R. V\'{a}zquez-Lorenzo,
``Higher dimensional Osserman metrics with non-nilpotent Jacobi operators",
{\it Geom. Dedicata \bf156} (2012) 151--163.

\bibitem{C-GR-VL-07}
E. Calvi\~{n}o-Louzao, E. Garc\'{\i}a-R\'{\i}o, and R. V\'{a}zquez-Lorenzo,
``Four-dimensional Osserman Ivanov Petrova metrics of neutral signature",
{\it Class. Quantum Grav. \bf 24} (2007), {2343--2355}.

\bibitem{CGM05}  M. Chaichi, E. Garc'a-R'o, and Y. Matsushita, 
``Curvature properties of four-dimensional Walker metrics",
{\it Classical Quantum Gravity \bf22} (2005), no. 3, 559--577.

\bibitem{CP08}  A. Chudecki and M. Przanowski, 
``From hyperheavenly spaces to Walker and Osserman spaces I",
{\it Classical Quantum Gravity \bf25} (2008), no. 14, 145010, 18 pp.

\bibitem{CP08a}  A. Chudecki and M. Przanowski, 
``From hyperheavenly spaces to Walker and Osserman spaces II",
{\it Classical Quantum Gravity \bf25} (2008),  no. 23, 235019, 22 pp.

\bibitem{Chi}
Q. S. Chi, ``A curvature characterization  of certain locally rank-one
symmetric spaces", {\it J. Diff. Geom.} {\bf 28} (1988), 187--202.


\bibitem{D09}  A. Derdzinski, ``Non-Walker self-dual neutral Einstein four-manifolds of Petrov type III",
{\it J. Geom. Anal. \bf19} (2009), 301--357.

\bibitem{DGV06} J.C. D\'{\i}az-Ramos,  E. Garc'a-R\'\i o, and R. V\'{a}zquez-Lorenzo,
``Four-dimensional Osserman metrics with
nondiagonalizable Jacobi operators", {\it J. Geom. Anal. \bf16}(2006), 39--52.

\bibitem{GGNV13} 
E. Garc\'{\i}a-R\'{\i}o, P.  Gilkey, S. Nik\v cevi\'c, and R. V\'{a}zquez-Lorenzo,
``Applications of Affine and Weyl Geometry", forthcoming from Morgan \& Claypool.


\bibitem{DGV1}
J. C. D\'{\i}az-Ramos, E. Garc\'\i a-R\'\i o, and R. V\'{a}zquez-Lorenzo,
``New examples of Osserman metrics with nondiagonalizable Jacobi operators",
{\it Differential Geom. Appl.} {\bf 24} (2006), 433--442.

\bibitem{GKVa}
E. Garc\'\i a-R\'\i o, D. N. Kupeli, and M. E. V\'{a}zquez-Abal, ``On a problem of
Osserman in Lorentzian geometry", {\it Differential Geom. Appl.}
{\bf 7} (1997), 85--100.

\bibitem{GKVV}
E. Garc\'{\i}a-R\'{\i}o, D. N. Kupeli, M. E. V\'{a}zquez-Abal, and  R. V\'{a}zquez-Lorenzo,
``Affine Osserman connections and their Riemann extensions",
{\it Differential Geom. Appl.} {\bf11} (1999), 145--153.

\bibitem{G-I-2}
P. Gilkey and R. Ivanova, ``The Jordan normal form of Osserman algebraic
curvature tensors", {\it Results Math.} {\bf 40} (2001), 192--204.

\bibitem{GSV95}
P. Gilkey, A. Swann, and L. Vanhecke,
``Isoparametric geodesic spheres and a conjecture
of Osserman concerning the Jacobi operator",
{\it Quart. J. Math. Oxford \bf 46} (1995), 299--320.
\bibitem{Ni1}
Y. Nikolayevsky, ``Osserman manifolds of dimension $8$", {\it
Manuscripta Math.} {\bf 115} (2004), 31--53.

\bibitem{Ni2}
Y. Nikolayevsky, ``Osserman conjecture in dimension $\neq 8, 16$", {\it
Math. Ann.} {\bf 331} (2005), 505--522.

\bibitem{Ni06} Y. Nikolayevsky,
``On Osserman manifolds of dimension 16", {\it Contemporary geometry and related topics}
Univ. Belgrade Fac. Math., Belgrade, (2006),  379--398.

\bibitem{Ni12} Y. Nikolayevsky, ``Conformally Osserman manifolds of dimension 16 and a 
Weyl-Schouten theorem for rank-one symmetric spaces", {\it Ann. Mat. Pura Appl. \bf19} (2012), 677-709

\bibitem{O90} R. Osserman,
``Curvature in the eighties",
{\it Amer. Math. Monthly} {\bf 97} (1990), 731--756.

\end{thebibliography}
\end{document}